%% file: master.tex
\documentclass{amsart}

\usepackage{preamble}

\usepackage{enumitem}
\setlist[enumerate,1]{label=(\alph*),ref=(\alph*)}

\theoremstyle{plain} \newtheorem*{thm}{Theorem}
 \newtheorem{theorem}{Theorem}[section]
\newtheorem{proposition}[theorem]{Proposition}
\newtheorem{lemma}[theorem]{Lemma}
\newtheorem{corollary}[theorem]{Corollary} \theoremstyle{definition}

\newtheorem{definition}[theorem]{Definition}
 \theoremstyle{remark}
\newtheorem{remark}[theorem]{Remark}

\begin{document}

\date{\today}
  
\title{An introduction to higher Auslander--Reiten theory}

\author[G.~Jasso]{Gustavo Jasso}
\address[Jasso]{Mathematisches Institut\\
  Rheinische Friedrich-Wilhelms-Universit\"at Bonn\\
  Endenicher Allee 60\\
  D-53115 Bonn\\
  Germany} \email{gjasso@math.uni-bonn.de}
\urladdr{https://gustavo.jasso.info}

\author[S.~Kvamme]{Sondre Kvamme}
\address[Kvamme]{Mathematisches Institut\\
  Rheinische Friedrich-Wilhelms-Universit\"at Bonn\\
  Endenicher Allee 60\\
  D-53115 Bonn\\
  Germany} \email{sondre@math.uni-bonn.de}

\begin{abstract}
  \input{abstract}
\end{abstract}



\thanks{\input{thanks}\input{thanks2}}

\maketitle

\input{content}

\bibliographystyle{alpha} \bibliography{mathscinet}

\end{document}

%% file: abstract.tex
This article consists of an introduction to Iyama's higher
Aus\-lan\-der--Reiten theory for Artin algebras from the viewpoint of
higher homological algebra. We provide alternative proofs of the basic
results in higher Auslander--Reiten theory, including the existence of
$d$-almost-split sequences in $d$-cluster-tilting subcategories,
following the approach to classical Auslander--Reiten theory due to
Auslander, Reiten, and Smal{\o}. We show that Krause's proof of
Auslander's defect formula can be adapted to give a new proof of the
defect formula for $d$-exact sequences. We use the defect formula to
establish the existence of morphisms determined by objects in
$d$-cluster-tilting subcategories.


%% file: thanks.tex
The authors thank Theo Raedschelders for bringing to their attention the article
\cite{MR1987342} and for participating in discussions which lead to the writing
of the present article. These thanks are extended to Julian K\"ulshammer and
Matthew Pressland for their comments on a previous version of this article. The
authors also wish to thank the anonymous referee for corrections and detailed
comments that improved the readability of this article.

%% file: thanks2.tex
The first named author wishes to thank Insitut Mittag-Leffler and the
organizers of program 1415s where a substantial part of this article
was written. The second named author's work was funded by the Bonn
International Graduate School of Mathematics.

%% file: content.tex
\section{Introduction}

Let $R$ be a commutative artinian ring. This article consists of an introduction
to Iyama's higher Auslander--Reiten theory for Artin $R$-algebras from the
viewpoint of higher homological algebra.

Auslander--Reiten theory was introduced in a series of articles between 1971 and
1978, see
\cite{MR0349747,MR0480688,MR0342505,MR0379599,MR0439881,MR0439882,MR0472919}.
Since its introduction, Auslander--Reiten theory has become a fundamental tool
for studying the representation theory of Artin algebras, see for example
\cite{MR2197389,MR1476671}.

The concept of an almost split sequence, central in Auslander--Reiten theory,
has its origin in the following fundamental theorem of Auslander.

\begin{thm}[(Auslander correspondence, see Section III.4 in
  \cite{auslander_representation_1971})]
  There is a one-to-one correspondence between the Morita-equivalence classes of
  Artin $R$-algebras $\Lambda$ of finite representation type and
  Morita-equivalence classes of Auslander $R$-algebras, that is Artin
  $R$-algebras $\Gamma$ satisfying
  \begin{equation*}
    \gldim\Gamma\leq 2\leq\operatorname{dom.dim}\Gamma
  \end{equation*}
  where $\operatorname{dom.dim}\Gamma$ denotes the dominant dimension of
  $\Gamma$ in the sense of \cite{MR0349740}. The correspondence is given by
  $M\mapsto\End_\Lambda(M)$, where $M$ is a representation generator of
  $\mmod\Lambda$.
\end{thm}

Let $\Lambda$ be an Artin $R$-algebra of finite representation type and $M$ a
representation generator of $\mmod\Lambda$. Suppose that $\Lambda$ is not
semi-simple. Then, the Artin $R$-algebra $\Gamma:=\End_\Lambda(M)$ has global
dimension $2$. Let $S$ be a simple $\Gamma$-module of projective dimension $2$
and
\begin{equation*}
  \begin{tikzcd}[column sep=small]
    0\rar&\Hom_\Lambda(M,M^{-2})\rar&\Hom_\Lambda(M,M^{-1})\rar&\Hom_\Lambda(M,M^0)\rar&S\rar&0
  \end{tikzcd}
\end{equation*}
a minimal projective resolution of $S$. By Yoneda's lemma, this resolution is
given by a left exact sequence
\begin{equation*}
  \begin{tikzcd}[column sep=small]
    0\rar&M^{-2}\rar&M^{-1}\rar&M^0
  \end{tikzcd}
\end{equation*}
of $\Lambda$-modules. It turns out that this sequence is exact and, moreover, an
almost split sequence. Conversely, an almost split sequence in $\mmod\Lambda$
corresponds to a minimal projective resolution of a simple $\Gamma$-module.

Higher Auslander--Reiten theory was introduced by Iyama in 2004 in
\cite{MR2298819} and \cite{MR2298820} (see also the survey article
\cite{MR2484730}). In addition to representation theory
\cite{MR2833569,MR2862066,MR3198750}, it has exhibited connections to
commutative algebra, commutative and non-commutative algebraic geometry, and
combinatorics, see for example
\cite{MR3357123,14090668,MR3144232,MR3251829,MR2984586}.

In complete analogy with the classical theory, higher Auslander--Reiten theory
can be motivated by the following theorem of Iyama. Let $d$ be a positive
integer. We recall that a (right) $\Lambda$-module $M$ is
\emph{$d$-cluster-tilting} if
\begin{align*}
  \add M &= \setP{X\in\mmod\Lambda}{\forall i\in\set{1,\dots,d-1}:\Ext_\Lambda^i(X,M)=0}\\
         &=\setP{Y\in\mmod\Lambda}{\forall i\in\set{1,\dots,d-1}:\Ext_\Lambda^i(M,Y)=0}.
\end{align*}
Note that a $1$-cluster-tilting $\Lambda$-module is precisely a representation
generator of $\mmod\Lambda$. The intrinsic properties of $d$-cluster-tilting
subcategories have been investigated in \cite{MR3021448,MR3519980,MR3544623}.

\begin{thm}[(Higher Auslander correspondence, see Theorem 0.2 in
  \cite{MR2298820})]
  There is a one-to-one correspondence between the equivalence classes of
  $d$-cluster-tilting modules for Artin $R$-algebras $\Lambda$ and
  Morita-equivalence classes of $d$-Auslander $R$-algebras, that is Artin
  $R$-algebras $\Gamma$ satisfying
  \begin{equation*}
    \gldim\Gamma\leq d+1\leq\operatorname{dom.dim}\Gamma
  \end{equation*}
  The correspondence is given by $M\mapsto\End_\Lambda(M)$, where $M$ is a
  $d$-cluster-tilting $\Lambda$-module.
\end{thm}

Analogously, the notion of a $d$-almost split sequence can be motivated as
follows. Let $\Lambda$ be an Artin $R$-algebra and $M$ a $d$-cluster-tilting
$\Lambda$-module. Suppose that $\Lambda$ is not semi-simple. Then, the Artin
$R$-algebra $\Gamma:=\End_\Lambda(M)$ has global dimension $d+1$. Let $S$ be a
simple $\Gamma$-module of projective dimension $d+1$ and
\begin{equation*}
  \begin{tikzcd}[column sep=tiny]
    0\rar&\Hom_\Lambda(M,M^{-(d+1)})\rar&\cdots\rar&\Hom_\Lambda(M,M^{-1})\rar&\Hom_\Lambda(M,M^0)\rar&S\rar&0
  \end{tikzcd}
\end{equation*}
a minimal projective resolution of $S$. By Yoneda's lemma, this resolution is
given by a sequence
\begin{equation*}
  \begin{tikzcd}[column sep=small]
    0\rar&M^{-(d+1)}\rar&\cdots\rar&M^{-1}\rar&M^0
  \end{tikzcd}
\end{equation*}
of $\Lambda$-modules. Such a sequence serves as the prototype for the definition
of a $d$-almost split sequence.

In Section \ref{sec:preliminaries} we motivate the introduction of
$d$-cluster-tilting subcategories of $\mmod\Lambda$ by investigating exact
sequences of the form
\begin{equation*}
  \begin{tikzcd}[column sep=small]
    0\rar&L\rar&M^1\rar&\cdots\rar&M^d\rar&N\rar&0
  \end{tikzcd}
\end{equation*}
whose terms lie in a $d$-rigid subcategory $\M$ of $\mmod\Lambda$, that is such
that for all $i\in\set{1,\dots,d-1}$ the equality $\Ext_\Lambda^i(\M,\M)=0$
holds. We call such sequences \emph{$d$-exact sequences}. In Section
\ref{sec:the_defect_of_an_exact_d-sequence} we modify Krause's proof of
Auslander's defect formula \cite{MR1987342} to give a new proof of the defect
formula for $d$-exact sequences. Following Auslander, Reiten, and Smal{\o}
\cite{MR1476671}, in Section \ref{sec:morphisms_determined_by_objects} we
establish the existence of morphisms determined by objects in
$d$-cluster-tilting subcategories. Finally, in Section \ref{sec:d-almost
  split_sequences}, we prove the existence of $d$-almost split sequences in
$d$-cluster-tilting subcategories as a consequence of the existence of morphisms
determined by objects.

\subsection*{Conventions and preliminaries}

Throughout the article we fix a positive integer $d$. We also fix a commutative
artinian ring $R$ and an Artin $R$-algebra $\Lambda$. We denote by
$\mmod\Lambda$ the abelian category of finite length, equivalently finitely
presented, right $\Lambda$-modules. Let $S_1,\dots, S_n$ be a complete set of
representatives of the isomorphism classes of the simple $R$-modules, and $I$
the injective envelope of the sum $S_1\oplus S_2\oplus\cdots \oplus S_n$. Recall
that
\begin{equation*}
  \begin{tikzcd}[column sep=small]
    D:=\Hom_R(-,I)\colon\mmod\Lambda\rar&\mmod\Lambda^{\op}
  \end{tikzcd}
\end{equation*}
is a duality. We denote by $\tau$ the Auslander--Reiten translation of
$\mmod\Lambda$. All subcategories of $\mmod\Lambda$ which we consider are
assumed to be full and closed under direct sums and direct summands; in
particular, they are additive subcategories of $\mmod\Lambda$. When convenient,
we write $(-,-)$ instead of $\Hom_\Lambda(-,-)$ and $^d(-,-)$ instead of
$\Ext_\Lambda^d(-,-)$. If $M\in\mmod\Lambda$, then we denote by $\add M$ the
smallest subcategory of $\mmod\Lambda$ containing $M$ and which is closed under
direct sums and direct summands.

We freely use classical concepts such as stable module categories, (co)syzygies,
and the Auslander--Bridger transpose. We refer the reader to \cite{MR1476671}
and \cite{MR2197389} for definitions. We denote the space of morphisms $M\to N$
which factor through a projective by $\mathcal{P}(M,N)$ and the projectively
stable module category by $\underline{\mmod}\Lambda$. Dually, we denote the
injectively stable module category by $\overline{\mmod}\Lambda$. We recall that
the \emph{(Jacobson) radical} of $\mmod\Lambda$ is the two-sided ideal defined
by
\begin{equation*}
  \rad\Hom_{\Lambda}(M,N):=\setP{f\in\Hom_{\Lambda}(M,N)}{1_N-f\circ g\text{ is invertible for all }g\in\Hom_{\Lambda}(N,M)},
\end{equation*}
see \cite{MR0170922} and Appendix A.3 in \cite{MR2197389} for further details.

Recall that a morphism $f\colon L\to M$ of $\Lambda$-modules is \emph{right
  minimal} if each morphism $g\colon L\to L$ such that $f\circ g=f$ is an
isomorphism. The notion of a morphism being \emph{left minimal} is defined
dually. Clearly, if $f\colon L\to M$ is a right minimal morphism and $g\colon
M\to N$ is a monomorphism, then the composite $g\circ f$ is right minimal. The
following observation is well known.

\begin{lemma}
  \label{minimal-jacobson} Let $L\xto{g}M\xto{f}N$ be a complex in
  $\mmod\Lambda$ such that the induced sequence of $R$-modules
  \begin{equation*}
    \begin{tikzcd}[column sep=small]
      \Hom_\Lambda(M,L)\rar&\Hom_\Lambda(M,M)\rar&\Hom_\Lambda(M,N)
    \end{tikzcd}
  \end{equation*} is exact. Then, $g$ is in the Jacobson radical of $\mmod\Lambda$ if and only if $f$ is right minimal.
\end{lemma}
\begin{proof}
  Suppose that $g$ is in the Jacobson radical of $\mmod\Lambda$. Let $h\colon
  M\to M$ be a morphism such that $f\circ h=f$. Given that $f\circ(h-1_M)=0$,
  the morphism $h-1_M$ factors through $g$, say as $g\circ i$. Since $g$ is in
  the Jacobson radical of $\mmod\Lambda$ the morphism $h=1_M+g\circ i$ is
  invertible, which is what we needed to show.

  Conversely, suppose that $f$ is right minimal. Let $i\colon M\to L$ be a
  morphism and set $h:=1_M+g\circ i$. Then, $f\circ h=f+f\circ(g\circ i)=f$.
  Since $f$ is right minimal, $h$ is an isomorphism. This shows that $g$ belongs
  to the Jacobson radical of $\mmod\Lambda$.
\end{proof}

Recall that a subcategory $\X$ of $\mmod\Lambda$ is \emph{contravariantly
  finite} if for every $M\in\mmod\Lambda$ there exists $X\in\X$ and a morphism
$f\colon X\to M$ such that for all $X'\in\X$ the induced sequence of $R$-modules
\begin{equation*}
  \begin{tikzcd}[column sep=small]
    \Hom_\Lambda(X',X)\rar&\Hom_\Lambda(X',M)\rar&0
  \end{tikzcd}
\end{equation*}
is exact. We call such a morphism $f$ a \emph{right $\X$-approximation of $M$}.
A right $\X$-approximation of $M$ is \emph{minimal} if it is a right minimal
morphism. We define \emph{covariantly finite subcategories} and \emph{(minimal)
  left approximations} dually. We say that $\X$ is \emph{functorially finite} if
it is both a contravariantly finite and a covariantly finite subcategory of
$\mmod\Lambda$.

A subcategory $\X$ of $\mmod\Lambda$ is said to be \emph{generating} if
$\Lambda\in\X$ (we remind the reader of our conventions on subcategories).
Similarly, $\X$ is \emph{cogenerating} if $D\Lambda\in\X$. Note that if $\X$ is
a generating contravariantly finite subcategory of $\mmod\Lambda$, then every
right $\X$-approximation $f\colon X\to M$ of $M$ is an epimorphism since the
projective cover of $M$ factors through $f$; dually, if $\X$ is a cogenerating
covariantly finite subcategory of $\mmod\Lambda$, then every left
$\X$-approximation is a monomorphism. We use these facts freely in the sequel.

\section{$d$-cluster-tilting subcategories}
\label{sec:preliminaries}

In this section we motivate the study of $d$-cluster-tilting subcategories of
$\mmod\Lambda$ from the viewpoint of $d$-homological algebra. For this, we first
introduce the notion of a $d$-exact sequence in a subcategory and establish
their basic properties.

\subsection{$d$-exact sequences}

We begin by investigating the properties of
exact sequences of length $d+2$ in $d$-rigid
subcategories of $\mmod\Lambda$.

\begin{definition}
  We say that a subcategory $\M$ of $\mmod\Lambda$ is \emph{$d$-rigid} if for
  each $i\in\set{1,\dots,d-1}$ the equality $\Ext_\Lambda^i(\M,\M)=0$ holds.
\end{definition}

Note that, by definition, if $d\geq2$ all short exact sequences in a $d$-rigid
subcategory are split. Nevertheless, the following simple observation motivates
the investigation of the properties of exact sequences of length $d+2$ in such
subcategories.

\begin{proposition}[(\cf Lemma 3.5 in \cite{MR2735750})]
  \label{d-rigid_long_exact} Let $\M$ be a $d$-rigid subcategory of
  $\mmod\Lambda$ and
  \begin{equation*}
    \begin{tikzcd}[column sep=small]
      0\rar&L\rar&M^1\rar&\cdots\rar&M^d\rar&N\rar&0
    \end{tikzcd}
  \end{equation*}
  an exact sequence in $\mmod\Lambda$ whose terms all lie in $\M$. Then, for
  each $X,Y\in\M$ there are exact sequences in $\mmod R$
  \begin{equation*}
    \begin{tikzcd}[column sep=tiny, row sep=tiny, ampersand
      replacement=\&]
      0\rar\&(X,L)\rar\&(X,M^1)\rar\&\cdots\rar\&(X,M^d)\rar\&(X,N)\rar\&^d(X,L)\rar\&^d(X,M^1)
    \end{tikzcd}
  \end{equation*}
  and
  \begin{equation*}
    \begin{tikzcd}[column sep=tiny, row sep=tiny, ampersand
      replacement=\&]
      0\rar\&(N,Y)\rar\&(M^d,Y)\rar\&\cdots\rar\&(M^1,Y)\rar\&(L,Y)\rar\&^d(N,Y)\rar\&^d(M^d,Y).
    \end{tikzcd}
  \end{equation*}
\end{proposition}
\begin{proof}
  We only construct the first sequence, the second sequence can be constructed
  dually. Since $\M$ is a $d$-rigid subcategory of $\mmod\Lambda$, for each
  $i\in\set{2,\dots,d-1}$ the cohomology of the complex
  \begin{equation*}
    \begin{tikzcd}[column sep=tiny]
      0\rar&\rar\Hom_\Lambda(X,L)\rar&\Hom_\Lambda(X,M^1)\rar&\cdots\rar&\Hom_\Lambda(X,M^d)\rar&\Hom_\Lambda(X,N)
    \end{tikzcd}
  \end{equation*}
  at $\Hom_\Lambda(X,M^i)$ is isomorphic to $\Ext_\Lambda^{i-1}(X,L)$ which
  vanishes by assumption (\cf remark 2.4.3 in \cite{MR1269324} in which
  $F$-acylic objects are discussed).

  Consider a short exact sequence $0\to K\to M^d\to N\to 0$ extending the map
  $M^d\to N$. Applying the functor $\Hom_\Lambda(X,-)$ yields an exact sequence
  \begin{equation*}
    \begin{tikzcd}[column sep=tiny]
      0\rar&\Hom_\Lambda(X,K)\rar&\Hom_\Lambda(X,M^d)\rar&\Hom_\Lambda(X,N)\rar&\Ext_\Lambda^1(X,K)\rar&0.
    \end{tikzcd}
  \end{equation*}
  Similarly, let $0\to L\to M^1\to K'\to 0$ be a short exact sequence extending
  the map $L\to M^1$. Applying the functor $\Hom_\Lambda(X,-)$ yields an exact
  sequence
  \begin{equation*}
    \begin{tikzcd}[column sep=small]
      0\rar&\Ext_\Lambda^{d-1}(X,K')\rar&\Ext_\Lambda^d(X,L)\rar&\Ext_\Lambda^d(X,M^1).
    \end{tikzcd}
  \end{equation*}
  By the dimension-shifting argument (see for example Exercise 2.4.3 in
  \cite{MR1269324}) there is an isomorphism between $\Ext_\Lambda^{d-1}(X,K')$
  and $\Ext_\Lambda^1(X,K)$. Thus, we obtain an exact sequence
  \begin{equation*}
    \begin{tikzcd}[column sep=tiny, row sep=small]
      \Hom_\Lambda(X,M^d)\rar&\Hom_\Lambda(X,N)\drar[two heads]\ar{rr}&&\Ext_\Lambda^d(X,L)\rar&\Ext_\Lambda^d(X,M^1).\\
      &&\Ext_\Lambda^1(X,K)\urar[tail]
    \end{tikzcd}
  \end{equation*}
  The claim follows.
\end{proof}

The study of the following class of exact sequences is motivated by
\ref{d-rigid_long_exact}.

\begin{definition}[(see Definition 2.2 and Definition 2.4 in \cite{MR3519980})]
  \label{d-exact_sequence} Let $\M$ be a subcategory of $\mmod\Lambda$. A
  complex
  \begin{equation*}
    \begin{tikzcd}[column sep=small]
      L\rar&M^1\rar&\cdots\rar&M^d\rar&N
    \end{tikzcd}
  \end{equation*}
  in $\M$ is a \emph{left $d$-exact sequence\footnote{We borrow this terminology
      from \cite{14092948}.} in $\M$} if for every $X\in\M$ the induced sequence
  of $R$-modules
  \begin{equation*}
    \begin{tikzcd}[column sep=small, row sep=tiny]
      0\rar&\Hom_\Lambda(X,L)\rar&\Hom_\Lambda(X,M^1)\rar&\cdots\rar&\Hom_\Lambda(X,M^d)\rar&\Hom_\Lambda(X,N)
    \end{tikzcd}
  \end{equation*}
  is exact. A \emph{right $d$-exact sequence} is defined dually. A
  \emph{$d$-exact sequence} is a sequence which is both a right $d$-exact
  sequence and a left $d$-exact sequence.
\end{definition}

\begin{remark}
  The reader should compare \ref{d-exact_sequence} with Definition 2.1 in
  \cite{MR2735750}.
\end{remark}

\begin{remark}
  Let $\M$ be a subcategory of $\mmod\Lambda$ and
  \begin{equation*}
    \begin{tikzcd}[column sep=small]
      L\rar{f}&M^1\rar&\cdots\rar&M^d\rar{g}&N
    \end{tikzcd}
  \end{equation*}
  a $d$-exact sequence in $\M$. By \ref{minimal-jacobson}, the fact that $f$ is
  left minimal is equivalent to the morphism $M^1\to M^2$ being in the Jacobson
  radical of $\mmod\Lambda$. Dually, $g$ is right minimal if and only if
  $M^{d-1}\to M^d$ belong to the Jacobson radical of $\mmod\Lambda$.
\end{remark}

According to \ref{d-rigid_long_exact}, exact sequences of length $d+2$ whose
terms all lie in a $d$-rigid subcategory of $\mmod\Lambda$ are $d$-exact
sequences in this subcategory. In fact, the following partial converse hold.

\begin{proposition}[(\cf Lemma 5.1 in \cite{MR3544623})]
  \label{exact_d-sequences_are_exact} Let $\M$ be a generating subcategory of
  $\mmod\Lambda$ and
  \begin{equation*}
    \begin{tikzcd}[column sep=small]
      \delta:&0\rar&L\rar&M^1\rar&\cdots\rar&M^d\rar&N
    \end{tikzcd}
  \end{equation*}
  a left $d$-exact sequence in $\M$. Then, $\delta$ is an exact sequence in
  $\mmod\Lambda$.
\end{proposition}
\begin{proof}
  The proof of Lemma 5.1 in \cite{MR3544623} carries over. Indeed, there is an
  isomorphism of complexes of $\Lambda$-modules
  \begin{equation*}
    \begin{tikzcd}[column sep=tiny]
      0\rar&\Hom_\Lambda(\Lambda,L)\rar\dar{\wr}&\Hom_\Lambda(\Lambda,M^1)\rar\dar{\wr}&\cdots\rar&\Hom_\Lambda(\Lambda,M^d)\rar\dar{\wr}&\Hom_\Lambda(\Lambda,N)\dar{\wr}\\
      0\rar&L\rar&M^1\rar&\cdots\rar&M^d\rar&N
    \end{tikzcd}
  \end{equation*}
  whose top row is exact by assumption.
\end{proof}

Recall that a morphism $f^\bullet\colon X^\bullet\to Y^\bullet$ is
\emph{null-homotopic} if there exists a family of morphisms $\setP{h^i\colon
  X^i\to Y^{i-1}}{i\in\ZZ}$ such that for each $i\in\ZZ$ the equation
\begin{equation*}
  f^i=h^{i+1}\circ d_X^i+d_Y^{i-1}\circ h^i
\end{equation*}
is satisfied. We need the following easy result which is analogous to the
comparison lemma for projective resolutions, see for example Theorem 2.2.6 in
\cite{MR1269324}.

\begin{lemma}
  \label{comparison_lemma} Let
  \begin{equation*}
    \begin{tikzcd}
      X^\bullet\dar{f^\bullet}&\cdots\rar&X^{-2}\rar\dar{f^{-2}}&X^{-1}\rar\dar{f^{-1}}&X^0\dar{0}\\
      Y^\bullet&\cdots\rar&Y^{-2}\rar&Y^{-1}\rar&Y^0
    \end{tikzcd}
  \end{equation*}
  be a morphism of complexes of $\Lambda$-modules such that for all $i\leq-1$
  the sequence of $R$-modules
  \begin{equation*}
    \begin{tikzcd}[column sep=tiny]
      \Hom_\Lambda(X^i,Y^{i-1})\rar&\Hom_\Lambda(X^i,Y^i)\rar&\Hom_\Lambda(X^i,Y^{i+1})
    \end{tikzcd}
  \end{equation*}
  is exact. Then, there exists a null-homotopy $h^\bullet$ of $f^\bullet$ such
  that the morphism $h^0\colon X^0\to Y^{-1}$ is the zero morphism.
\end{lemma}
\begin{proof}
  See for example the Comparison Lemma 2.1 in \cite{MR3519980}.
\end{proof}

We recall that a complex $X^\bullet$ of $\Lambda$-modules is \emph{contractible}
if its identity morphism is null-homotopic. The following straightforward
result, well known for short exact sequences, is an important property of
$d$-exact sequences.

\begin{proposition}[(see Proposition 2.6 \cite{MR3519980})]
  \label{contractible_iff_split_morphisms} Let $\M$ be a subcategory of
  $\mmod\Lambda$ and
  \begin{equation*}
    \begin{tikzcd}[column sep=small]
      \delta\colon&0\rar&L\rar{f}&M^1\rar&\cdots\rar&M^d\rar{g}&N\rar&0
    \end{tikzcd}
  \end{equation*}
  a $d$-exact sequence in $\M$. Then, the following statements are equivalent.
  \begin{enumerate}
  \item\label{it:contractible} $\delta$ is a contractible complex.
  \item\label{it:split_mono} $f$ is a split monomorphism.
  \item\label{it:split_epi} $g$ is a split epimorphism.
  \end{enumerate}
\end{proposition}
\begin{proof}
  It is clear that \ref{it:contractible} implies both \ref{it:split_mono} and
  \ref{it:split_epi}. We only show that \ref{it:split_epi} implies
  \ref{it:contractible}. Since $g$ is a split epimorphism, for all $X\in\M$ the
  complex
  \begin{equation*}
    \begin{tikzcd}[column sep=tiny]
      0\rar&\Hom_\Lambda(X,L)\rar&\Hom_\Lambda(X,M^1)\rar&\cdots\rar&\Hom_\Lambda(X,M^d)\rar&\Hom_\Lambda(X,N)\rar&0
    \end{tikzcd}
  \end{equation*}
  is exact. In particular, the morphism of complexes
  \begin{equation*}
    \begin{tikzcd}
      0\dar\rar&L\rar{f}\dar{1}&M^1\rar\dar{1}&\cdots\rar&M^d\rar{g}\dar{1}&N\rar\dar{1}&0\dar\\
      0\rar&L\rar{f}&M^1\rar&\cdots\rar&M^d\rar{g}&N\rar&0
    \end{tikzcd}
  \end{equation*} satisfies the hypothesis in \ref{comparison_lemma} and, therefore, the identity morphism of $\delta$ is null-homotopic.
\end{proof}

Pullback and pushout diagrams play an important role in homological algebra. We
consider the following higher analogues of these concepts.

\begin{definition}(see Definition 2.11 in \cite{MR3519980})
  \label{pullback} Let $\M$ be a subcategory of $\mmod\Lambda$. A morphism of
  complexes of $\Lambda$-modules of the form
  \begin{equation*}
    \begin{tikzcd}
      M^0\rar{f^0}\dar{u^0}&M^1\rar{f^1}\dar{u^1}&\cdots\rar{f^{d-1}}&M^d\dar{u^d}\\
      N^0\rar{g^0}&N^1\rar{g^1}&\cdots\rar{g^{d-1}}&N^d
    \end{tikzcd}
  \end{equation*}
  is a \emph{$d$-pullback diagram in $\M$} (resp. \emph{$d$-pushout diagram in
    $\M$}) if the mapping cone
  \begin{equation*}
    \begin{tikzcd}[column sep=huge, ampersand replacement=\&]
      M^0\rar{\begin{bmatrix} -f^0\\u^0
        \end{bmatrix}}\&M^1\oplus N^0\rar{\begin{bmatrix} -f^1&0\\u^1&g^0
        \end{bmatrix}}\&\cdots\rar{\begin{bmatrix} -f^{d-1}&0\\u^{d-1}&g^{d-2}
        \end{bmatrix}}\&M^d\oplus N^{d-1}\rar{\begin{bmatrix} u^d&g^{d-1}
        \end{bmatrix}}\&N^d
    \end{tikzcd}
  \end{equation*}
  is a left $d$-exact sequence (resp. right $d$-exact sequence) in $\M$.
\end{definition}

\begin{remark}
  Let $\M$ be a subcategory of $\mmod\Lambda$. The reader can readily verify
  that a morphism of complexes of $\Lambda$-modules of the form
  \begin{equation*}
    \begin{tikzcd}[column sep=small]
      N^0\rar\dar&\cdots\rar&N^{i-1}\rar\dar&N^i\rar\dar&0\rar\dar&\cdots\rar&0\rar\dar&0\dar\\
      0\rar&\cdots\rar&0\rar&N^{i+1}\rar&N^{i+2}\rar&\cdots\rar&N^{d+1}\rar&N^{d+2}
    \end{tikzcd}
  \end{equation*}
  is a $d$-pullback diagram (resp. $d$-pullback) in $\M$ if and only if
  \begin{equation*}
    \begin{tikzcd}[column sep=small]
      N^0\rar&\cdots\rar&N^d\rar&N^{d+1}\rar&N^{d+2}
    \end{tikzcd}
  \end{equation*}
  is a left $d$-exact sequence (resp. right $d$-exact sequence) in $\M$.
\end{remark}

The following result is analogous to the classical statement that pullbacks
preserve kernels.

\begin{proposition}[(\cf Proposition 4.8 \cite{MR3519980})]
  \label{d-pullback_left_exact_d-sequence} Let $\M$ be a subcategory of
  $\mmod\Lambda$ and consider a diagram
  \begin{equation*}
    \begin{tikzcd}
      &U^1\rar\dar&\cdots\rar&U^d\rar\dar&V\dar\\
      L\rar&M^1\rar&\cdots\rar&M^d\rar&N
    \end{tikzcd}
  \end{equation*}
  where the rectangle is a $d$-pullback diagram in $\M$ and whose bottom row is
  a left $d$-exact sequence. Then, the diagram above extends uniquely to a
  morphism of complexes of the form
  \begin{equation*}
    \begin{tikzcd}
      L\dar[equals]\rar&U^1\rar\dar&\cdots\rar&U^d\rar\dar&V\dar\\
      L\rar&M^1\rar&\cdots\rar&M^d\rar&N
    \end{tikzcd}
  \end{equation*}
  whose top row is also a left $d$-exact sequence.
\end{proposition}
\begin{proof}
  Consider the standard sequence in the abelian category of complexes of
  $\Lambda$-modules
  \begin{equation*}
    \begin{tikzcd}[column sep=small]
      U^\bullet\dar{\varphi}& 0\rar\dar&U^1\rar\dar&U^2\rar\dar&\cdots\rar&U^d\rar\dar&V\dar\\
      M^\bullet\dar& 0\rar\dar&M^1\rar\dar&M^2\rar\dar&\cdots\rar&M^d\rar\dar&N\dar\\
      C^\bullet(\varphi)\dar& U^1\rar\dar&U^2\oplus M^1\rar\dar&U^3\oplus M^2\dar\rar&\cdots\rar&V\oplus M^d\rar\dar&N\dar\\
      U^\bullet[1]& U^1\rar&U^2\rar&U^3\rar&\cdots\rar&V\rar&0.
    \end{tikzcd}
  \end{equation*}
  Let $M\in\M$. Then, applying the functor $\Hom_\Lambda(M,-)$ to this diagram
  yields a commutative diagram
  \begin{equation*}
    \resizebox{\textwidth}{!}{
      \begin{tikzcd}[column sep=tiny, ampersand replacement=\&]
        \&\Hom_\Lambda(M,M^1)\rar\dar\&\Hom_\Lambda(M,M^2)\rar\dar\&\cdots\rar\&\Hom_\Lambda(M,M^d)\rar\dar\&\Hom_\Lambda(M,N)\dar\\
        \Hom_\Lambda(M,U^1)\rar\dar\&\Hom_\Lambda(M,U^2\oplus M^1)\rar\dar\&\Hom_\Lambda(M,U^3\oplus M^2)\dar\rar\&\cdots\rar\&\Hom_\Lambda(M,V\oplus M^d)\rar\dar\&\Hom_\Lambda(M,N)\\
        \Hom_\Lambda(M,U^1)\rar\&\Hom_\Lambda(M,U^2)\rar\&\Hom_\Lambda(M,U^3)\rar\&\cdots\rar\&\Hom_\Lambda(M,V).
      \end{tikzcd}}
  \end{equation*}
  whose vertical columns are split short exact sequences. Since the two top rows
  are exact, a diagram chase shows that the bottom row is also exact.
  
  By the definition of a $d$-pullback diagram, there exists a morphism $L\to
  U^1$ rendering the diagram
  \begin{equation*}
    \begin{tikzcd}
      L\ar[bend right]{ddr}\ar[bend left]{drr}{0}\drar[dotted]\\
      &U^1\rar\dar&U^2\dar\\
      &M^1\rar&M^2
    \end{tikzcd}
  \end{equation*}
  commutative. Since $L\to M^1$ is a monomorphism, $L\to U^1$ is also a
  monomorphism. Let $W\to U^1$ be a morphism such that the diagram
  \begin{equation*}
    \begin{tikzcd}
      &W\dar\drar{0}\\
      L\rar&U^1\rar&U^2.
    \end{tikzcd}
  \end{equation*}
  commutes. It follows that there exists a morphism $W\to L$ such that the
  diagram
  \begin{equation*}
    \begin{tikzcd}
      &W\dar\drar{0}\ar[dotted]{ddl}\\
      L\rar\dar[equals]&U^1\rar\dar&U^2\dar\\
      L\rar&M^1\rar&M^2
    \end{tikzcd}
  \end{equation*}
  commutes. Using the fact that $U^1\to U^2\oplus M^1$ is a monomorphism it can
  be easily verified that the diagram
  \begin{equation*}
    \begin{tikzcd}
      &W\dar\drar{0}\dlar[dotted]\\
      L\rar\dar[equals]&U^1\rar\dar&U^2\dar\\
      L\rar&M^1\rar&M^2
    \end{tikzcd}
  \end{equation*}
  commutes. This finishes the proof.
\end{proof}

\subsection{$d$-cluster-tilting subcategories}

The following proposition provides motivation for introducing the class of
$d$-cluster-tilting subcategories of $\mmod\Lambda$, \cf Theorem 2.2.3 in
\cite{MR2298819} (see Proposition 3.17 in \cite{MR3519980} for a proof).

\begin{proposition}
  \label{contrafinite_d-rigid_almost-d-kernels} Let $\M$ be a $d$-rigid
  subcategory of $\mmod\Lambda$. Consider a commutative diagram
  \begin{equation*}
    \begin{tikzcd}[column sep=small]
      0\rar&C^1\rar\dar[equals]&M^1\rar&\cdots\rar&M^{d-1}\rar\dar&M^d\rar&N\\
      &C^1\urar&&\cdots&C^d\urar
    \end{tikzcd}
  \end{equation*}
  such that $M^d\to N$ is a morphism in $\M$ and such that following statements
  hold:
  \begin{itemize}
  \item For each $i\in\set{1,\dots,d}$ the morphism $C^i\to M^i$ is a kernel of
    $M^i\to M^{i+1}$ (by convention, $M^{d+1}:=N$).
  \item For each $i\in\set{1,\dots,d-1}$ the morphism $M^i\to C^{i+1}$ is a
    surjective right $\M$-approximation of $C^{i+1}$.
  \end{itemize}
  Then, for each $i\in\set{1,\dots,d-1}$ the equality $\Ext_\Lambda^i(\M,C^1)=0$
  holds. Moreover, for each $X\in\M$ the induced sequence of $R$-modules
  \begin{equation*}
    \begin{tikzcd}[column sep=tiny, row sep=tiny]
      0\rar&\Hom_\Lambda(X,C^1)\rar&\Hom_\Lambda(X,M^1)\rar&\cdots\rar&\Hom_\Lambda(X,M^d)\rar&\Hom_\Lambda(X,N)
    \end{tikzcd}
  \end{equation*}
  is exact.
\end{proposition}

The following definition is motivated by
\ref{contrafinite_d-rigid_almost-d-kernels} and its dual.

\begin{definition}[(see Definition 2.2 in \cite{MR2298819})]
  Let $\M$ be a functorially finite subcategory of $\mmod\Lambda$. We call $\M$
  a \emph{$d$-cluster-tilting subcategory} if
  \begin{align*}
    \M=&\setP{X\in\mmod\Lambda}{\forall k\in\set{1,\dots,d-1}:\Ext_\Lambda^k(X,\M)=0}\\
    =&\setP{Y\in\mmod\Lambda}{\forall
       k\in\set{1,\dots,d-1}:\Ext_\Lambda^k(\M,Y)=0}.
  \end{align*}
\end{definition}

\begin{remark}
  Note that $\mmod\Lambda$ itself is its unique $1$-cluster-tilting subcategory.
  More generally, let $\M$ be a $d$-cluster-tilting subcategory of
  $\mmod\Lambda$. By definition, $\M$ is a both a generating and a cogenerating
  subcategory of $\mmod\Lambda$. In the spirit of the Morita--Tachikawa
  correspondence \cite{MR0349740}, the study of $d$-cluster-tilting
  subcategories can be regarded as the study of a particularly nice class of
  generating-cogenerating subcategories of $\mmod\Lambda$, see for example
  \cite{MR2298820,160200127}.
\end{remark}

One of the defining properties of abelian categories is that every epimorphism
is the cokernel of its kernel. The following result shows that
$d$-cluster-tilting subcategories of $\mmod\Lambda$ satisfy an analogous
property, \cf Theorem 3.16 and Proposition 3.18 in \cite{MR3519980}.

\begin{proposition}\label{d-ct_a2} Let
  $\M$ be a $d$-cluster-tilting subcategory of $\mmod\Lambda$ and
  \begin{equation*}
    \begin{tikzcd}[column sep=small]
      \delta\colon&0\rar&L\rar&M^1\rar&\cdots\rar&M^d\rar{g}&N
    \end{tikzcd}
  \end{equation*}
  a left $d$-exact sequence in $\M$ such that $g$ is an epimorphism. Then,
  $\delta$ is a $d$-exact sequence.
\end{proposition}
\begin{proof}
  Since $g$ is an epimorphism, by \ref{exact_d-sequences_are_exact} the sequence
  \begin{equation*}
    \begin{tikzcd}[column sep=small]
      0\rar&L\rar&M^1\rar&\cdots\rar&M^d\rar{g}&N\rar&0
    \end{tikzcd}
  \end{equation*}
  is exact. Since $\M$ is in particular a $d$-rigid subcategory, the claim
  follows from \ref{d-rigid_long_exact}.
\end{proof}

\begin{proposition}
  \label{minimal}
  Let $\M$ be a $d$-cluster-tilting subcategory of $\mmod\Lambda$ and $g\colon
  M^d\to N$ a morphism in $\M$. 
  Then, there exists a left $d$-exact sequence in $\M$ of the
    form
    \begin{equation*}
      \begin{tikzcd}[column sep=small]
        0\rar&L\rar{f}&M^1\rar&\cdots\rar&M^{d-1}\rar&M^d\rar{g}&N.
      \end{tikzcd}
    \end{equation*}
    Moreover, for all $i\in\set{1,\dots,d-1}$ the morphism $M^i\to M^{i+1}$ can
    be chosen to be right minimal.
\end{proposition}
\begin{proof}
  This is immediate from the fact that $\M$ is a contravariantly
  finite subcategory, \ref{contrafinite_d-rigid_almost-d-kernels}, and the
  definition of a $d$-cluster-tilting subcategory. For the right minimality,
  choose the right $\M$-approximations $M^i\to C^{i+1}$ used to construct the
  left $d$-exact sequence to be right minimal.
\end{proof}

Combining \ref{minimal}, \ref{d-ct_a2}, and their duals, gives the following
result.

\begin{proposition}
  Let $\M$ be a $d$-cluster-tilting subcategory of $\mmod\Lambda$. Then, $\M$ is
  a $d$-abelian category in the sense of Definition 3.1 in \cite{MR3519980}.
\end{proposition}

We conclude this section with a couple of technical results regarding
$d$-pullback diagrams and $d$-pushout diagrams in $d$-cluster-tilting
subcategories of $\mmod\Lambda$.

\begin{proposition}[(see Theorem 3.8 in \cite{MR3519980})]
  \label{existence_d-pullback} Let $\M$ be a $d$-cluster-tilting subcategory of
  $\mmod\Lambda$,
  \begin{equation*}
    \begin{tikzcd}[column sep=small]
      M^1\rar&\cdots\rar&M^d\rar&M
    \end{tikzcd}
  \end{equation*}
  a complex in $\M$, and $f\colon N\to M$ a morphism in $\M$. Then, there exists
  a $d$-pullback diagram in $\M$ of the form
  \begin{equation*}
    \begin{tikzcd}
      N^1\rar\dar&\cdots\rar&N^d\rar\dar&N\dar{f}\\
      M^1\rar&\cdots\rar&M^d\rar&M.
    \end{tikzcd}
  \end{equation*}
\end{proposition}
\begin{proof}
  We shall construct a commutative diagram
  \begin{equation*}
    \begin{tikzcd}[column sep=small]
      &N^2\drar\ar{rr}&&\cdots&&\cdots\drar\ar{rr}&&N^d\drar\ar{rr}&&N\drar\\
      N^1=C^1\urar\drar&&C^2\urar\drar&&\cdots&&C^{d-1}\urar\drar&&C^d\urar\drar&&M\\
      &M^1\urar\ar{rr}&&\cdots&&\cdots\urar\ar{rr}&&M^{d-1}\urar\ar{rr}&&M^d\urar
    \end{tikzcd}
  \end{equation*}
  as follows. Construct a pullback diagram
  \begin{equation*}
    \begin{tikzcd}
      C^d\rar\dar&N\dar{f}\\
      M^d\rar&M.
    \end{tikzcd}
  \end{equation*}
  Since the solid diagram
  \begin{equation*}
    \begin{tikzcd}
      M^{d-1}\drar\rar[dotted]\ar[bend left]{rr}{0}&C^d\rar\dar&N\dar{f}\\
      &M^d\rar&M
    \end{tikzcd}
  \end{equation*}    
  commutes, the dotted morphism rendering the diagram commutative exists. Let
  $N^d\to C^d$ be a right $\M$-approximation of $C^d$. Construct a pullback
  diagram
  \begin{equation*}
    \begin{tikzcd}
      C^{d-1}\rar\dar&N^d\dar\\
      M^{d-1}\rar&C^d
    \end{tikzcd}
  \end{equation*}
  and iterate the previous argument. By construction, the resulting diagram
  induces a commutative diagram
  \begin{equation*}
    \begin{tikzcd}[column sep=small]
      0\rar&C^1\rar\dar[equals]&N^2\oplus M^1\rar&\cdots\rar&N^d\oplus M^{d-1}\rar\dar&N\oplus M^d\rar&M\\
      &C^1\urar&&\cdots&C^d\urar
    \end{tikzcd}
  \end{equation*}
  which satisfies the assumptions in \ref{contrafinite_d-rigid_almost-d-kernels}
  whence $C^1$ lies in $\M$ and
  \begin{equation*}
    \begin{tikzcd}[column sep=small]
      0\rar&C^1\rar&N^2\oplus M^1\rar&\cdots\rar&N^d\oplus M^{d-1}\rar&N\oplus
      M^d\rar{g}&M
    \end{tikzcd}
  \end{equation*}
  is a left $d$-exact sequence in $\M$. This sequence is precisely the mapping
  cone of the morphism of complexes
  \begin{equation*}
    \begin{tikzcd}
      N^1=C^1\rar\dar&\cdots\rar&N^d\rar\dar&N\dar{f}\\
      M^1\rar&\cdots\rar&M^d\rar&M.
    \end{tikzcd}
  \end{equation*}
  The claim follows.
\end{proof}

\begin{proposition}[(see Proposition 4.8 in \cite{MR3519980})]
  \label{props_d-pullback} Let $\M$ be a $d$-cluster-tilting subcategory of
  $\mmod\Lambda$. Consider a morphism of complexes in $\M$ of the form
  \begin{equation*}
    \begin{tikzcd}
      0\rar&L\rar\dar[equals]&U^1\rar\dar&\dots\rar&U^d\rar{g'}\dar&V\dar\\
      0\rar&L\rar&M^1\rar&\cdots\rar&M^d\rar{g}&N\rar&0
    \end{tikzcd}
  \end{equation*}
  whose bottom row is a $d$-exact sequence. Then, the following statements are
  equivalent:
  \begin{enumerate}
  \item\label{it:i} The top row is a $d$-exact sequence.
  \item\label{it:ii} The rightmost $d$ squares form a $d$-pullback diagram.
  \item\label{it:iii} The rightmost $d$ squares form a $d$-pullback diagram and a
    $d$-pushout diagram.
  \end{enumerate}
\end{proposition}
\begin{proof}
  We show that \ref{it:ii} implies \ref{it:iii}. By assumption
  \begin{equation*}
    \begin{tikzcd}[column sep=small]
      0\rar&U^1\rar&U^2\oplus M^1\rar&\cdots\rar&V\oplus M^d\rar&N
    \end{tikzcd}
  \end{equation*}
  is a left $d$-exact sequence in $\M$. Since the morphism $V\oplus M^d\to N$ is
  an epimorphism, the claim follows from \ref{d-ct_a2}.

  Now we show that statement \ref{it:iii} implies statement \ref{it:i}. By
  \ref{d-pullback_left_exact_d-sequence} the top row of the diagram is a left
  $d$-exact sequence. Hence, in view of \ref{d-ct_a2}, it is enough to show that
  $g'$ is an epimorphism. For this, let $V\to W$ be a morphism such that the
  composite $U^d\xto{g'}V\to W$ vanishes. Since the solid diagram
  \begin{equation*}
    \begin{tikzcd}
      U^d\rar{g'}\dar&V\dar\ar[bend left]{ddr}\\
      M^d\rar{g}\ar[bend right]{drr}{0}&N\drar[dotted]\\&&W
    \end{tikzcd}
  \end{equation*}
  commutes and the inside square is part of a $d$-pushout diagram, the dotted
  morphism $N\to W$ rendering the diagram commutative exists. Since $g$ is an
  epimorphism, and the composite $M^d\to N\to W$ vanishes, the morphism $N\to W$
  is the zero morphism. Therefore the morphism $V\to W$ is also zero and we
  conclude that $g'$ is an epimorphism. The fact that statement \ref{it:i}
  implies statement \ref{it:ii} can be shown using an argument similar to that
  used in the proof of \ref{d-pullback_left_exact_d-sequence}. We leave the
  details to the reader.
\end{proof}

\section{The defect of a $d$-exact sequence}
\label{sec:the_defect_of_an_exact_d-sequence}

In this section we introduce the defect of a $d$-exact sequence, which is
analogous to the defect of a short exact sequence. We prove a higher version of
Auslander's defect formula using a minor modification of Krause's proof of the
classical formula, \cf \cite{MR1987342}.

\begin{definition}(see Section IV.4 in \cite{MR1476671}) Let $\M$ be a
  subcategory of $\mmod\Lambda$ and
  \begin{equation*}
    \begin{tikzcd}[column sep=small]
      \delta\colon&L\rar&M^1\rar&\cdots\rar&M^d\rar&N
    \end{tikzcd}
  \end{equation*}
  a $d$-exact sequence in $\M$. We define $\delta^*$, the \emph{contravariant
    defect of $\delta$}, by the exact sequence of functors
  \begin{equation*}
    \begin{tikzcd}[column sep=small]
      \Hom_\Lambda(-,M^d)\rar&\Hom_\Lambda(-,N)\rar&\delta^*\rar&0.
    \end{tikzcd}
  \end{equation*}
  Dually, the \emph{covariant defect of $\delta$}, denoted by $\delta_*$, is
  defined by the exact sequence of functors
  \begin{equation*}
    \begin{tikzcd}[column sep=small]
      \Hom_\Lambda(M^1,-)\rar&\Hom_\Lambda(L,-)\rar&\delta_*\rar&0.
    \end{tikzcd}
  \end{equation*}  
\end{definition}

\begin{remark}
  Let $\M$ be a subcategory of $\mmod\Lambda$. Let $\dEx(\M)$ be the category
  whose objects are the $d$-exact sequences in $\M$ and whose morphisms are given by
  morphisms of complexes. Clearly, the association $(\delta,X)\mapsto
  \delta^*(X)$ induces a functor
  \begin{equation*}
    \dEx(\M)\times \M^\op\to\mmod R
  \end{equation*}
  which is additive in each variable. Dually,
  the association $(\delta,X)\mapsto \delta_*(X)$ induces a functor
  \begin{equation*}
    \dEx(\M)^\op\times \M\to\mmod R
  \end{equation*}
  which is additive in each variable.
\end{remark}

The following observation shows that the defect is well defined up to homotopy
equivalence of $d$-exact sequences (considered as complexes).

\begin{proposition}
  Let $\M$ be a subcategory of $\mmod\Lambda$. The association
  $(\delta,X)\mapsto\delta^*(X)$ yields a bifunctor
  \begin{equation*}
    \KdEx(\M)\times\underline{\M}^\op\to\mmod R,
  \end{equation*}
  where $\KdEx(\M)$ denotes the category of $d$-exact sequences in $\M$ and
  homotopy classes of cochain morphisms between them.
\end{proposition}
\begin{proof}
  Let
  \begin{equation*}
    \begin{tikzcd}
      \delta\dar{\varphi}&L\rar\dar&M^1\rar\dar&\cdots\rar&M^d\rar\dar&N\dar\\
      \varepsilon&U\rar&V^1\rar&\cdots\rar&V^d\rar&W
    \end{tikzcd}
  \end{equation*}
  a null-homotopy between $d$-exact sequences in $\M$. We need to verify that
  the morphism $\varphi^*\colon\delta^*\to\varepsilon^*$ vanishes. By the
  functoriality and additivity of $\Hom_\Lambda(-,-)$, for each $M\in\M$ the
  induced morphism of complexes
  \begin{equation*}
    \begin{tikzcd}[column sep=tiny]
      0\rar&\Hom_\Lambda(M,L)\rar\dar&\Hom_\Lambda(M,M^1)\rar\dar&\cdots\rar&\Hom_\Lambda(M,M^d)\rar\dar&\Hom_\Lambda(M,N)\dar\rar&0\\
      0\rar&\Hom_\Lambda(M,U)\rar&\Hom_\Lambda(M,V^1)\rar&\cdots\rar&\Hom_\Lambda(M,V^d)\rar&\Hom_\Lambda(M,W)\rar&0
    \end{tikzcd}
  \end{equation*}
  is also a null-homotopy whence the induced morphism in cohomology vanishes.
  Since $\delta^*(M)$ is the cohomology of the top row at $\Hom_\Lambda(M,N)$
  and $\varepsilon^*(M)$ is the cohomology of the bottom row at
  $\Hom_\Lambda(M,W)$, we conclude that $\varphi^*(M)$ vanishes. Finally, it is
  clear that $\delta^*$ vanishes on projectives.
\end{proof}

We denote Heller's syzygy functor by $\Omega\colon\smod\Lambda\to\smod\Lambda$.
Following \cite{MR2298819}, we consider the functor
\begin{equation*}
  \begin{tikzcd}[column sep=small, row sep=tiny]
    \Tr_d:=\Tr\Omega^{d-1}\colon\smod\Lambda\rar&\smod\Lambda^\op
  \end{tikzcd}
\end{equation*}
which is a higher analogue of the Auslander--Bridger transposition. Note that if
$M\in\mmod\Lambda$ has projective dimension strictly less than $d$, then
$\Tr_dM=0$.

\begin{proposition}[(see Proposition 1.1.1 in \cite{MR2298819})]
  \label{Tr_equivalence} Let $\M$ be a generating $d$-rigid subcategory of
  $\mmod\Lambda$. Then, $\Tr_d\colon\sM\to\smod\Lambda^\op$ is fully
  faithful.
\end{proposition}
\begin{proof}
  Let $M,N\in\M$ and consider commutative diagram
  \begin{equation*}
    \begin{tikzcd}
      P^d\rar\dar&\cdots\rar&P^1\rar\dar&P^0\rar\dar&M\rar\dar{f}&0\\
      Q^d\rar&\cdots\rar&Q^1\rar&Q^0\rar&N\rar&0
    \end{tikzcd}
  \end{equation*}
  where the rows give minimal projective resolutions of $M$ and $N$. Applying
  the functor $\Hom_\Lambda(-,\Lambda)$ to this diagram yields a commutative
  diagram
  \begin{equation*}
    \begin{tikzcd}[column sep=small]
      \Hom_\Lambda(Q^0,\Lambda)\rar\dar&\cdots\rar&\Hom_\Lambda(Q^{d-1},\Lambda)\rar\dar&\Hom_\Lambda(Q^d,\Lambda)\rar\dar&\Tr_dN\dar{\Tr_df}\rar\dar&0\\
      \Hom_\Lambda(P^0,\Lambda)\rar&\cdots\rar&\Hom_\Lambda(P^{d-1},\Lambda)\rar&\Hom_\Lambda(P^d,\Lambda)\rar&\Tr_dM\rar&0
    \end{tikzcd}
  \end{equation*}
  whose rows give minimal projective resolutions of $\Tr_d M$ and $\Tr_d N$.
  Indeed, for $i\in\set{1,\dots,d-1}$ the cohomology of the top row at
  $\Hom_\Lambda(Q^i,\Lambda)$ is isomorphic to $\Ext_\Lambda^i(N,\Lambda)$ which
  vanishes by assumption. Similarly, for $i\in\set{1,\dots,d-1}$ the cohomology
  of the top row at $\Hom_\Lambda(P^i,\Lambda)$ is isomorphic to
  $\Ext_\Lambda^i(M,\Lambda)$ which vanishes by assumption. Since
  $\Hom_\Lambda(-,\Lambda)$ restricts to a duality between projective
  $\Lambda$-modules and projective $\Lambda^\op$-modules, the claim follows.
\end{proof}

We recall the following homological result which provides some motivation for
the introduction of the higher transpose.

\begin{proposition}[(see Proposition 1.1.3 in \cite{MR2298819})]
  \label{Tor-Ext_iso} Let $\M$ be a generating $d$-rigid subcategory of
  $\mmod\Lambda$. Then, for every $X\in\M$, every $M\in\mmod\Lambda$, and each
  $i\in\set{1,\dots,d-1}$ there are functorial isomorphisms
  \begin{equation*}
    \Tor_{d-i}^\Lambda(M,\Tr_dX)\cong\Ext_\Lambda^i(X,M)
  \end{equation*}
\end{proposition}
\begin{proof}
  The proof is straightforward. We refer the reader to Proposition 1.1.3 in
  \cite{MR2298819} for details.
\end{proof}

The following result is an easy consequence of \ref{Tor-Ext_iso}.

\begin{corollary}
  \label{tensor_Trd_is_right_d-exact} Let $\M$ be a generating-cogenerating
  $d$-rigid subcategory of $\mmod\Lambda$ and
  \begin{equation*}
    \begin{tikzcd}[column sep=small]
      \delta\colon&0\rar&L\rar&M^1\rar&\cdots\rar&M^d\rar&N\rar&0
    \end{tikzcd}
  \end{equation*}
  a $d$-exact sequence in $\M$. Then, for each $X\in\M$ the induced sequence
  \begin{equation*}
    \begin{tikzcd}[column sep=small]
      L\otimes_\Lambda\Tr_dX\rar&M^1\otimes_\Lambda\Tr_dX\rar&\cdots\rar&M^d\otimes_\Lambda\Tr_dX\rar&N\otimes_\Lambda\Tr_dX\rar&0
    \end{tikzcd}
  \end{equation*}
  is exact in $\mmod\Lambda$.
\end{corollary}
\begin{proof}
  In view of \ref{Tor-Ext_iso}, for each $i\in\set{1,\dots,d-1}$ and for all
  $M'\in\M$ the equality
  $\Tor_{d-i}^\Lambda(M',\Tr_dX)\cong\Ext_\Lambda^i(X,M')=0$ holds. By an
  argument similar to that used in the proof of \ref{d-rigid_long_exact} and
  \ref{Tor-Ext_iso} the cohomology of the complex
  \begin{equation*}
    \begin{tikzcd}[column sep=small]
      L\otimes_\Lambda\Tr_dX\rar&M^1\otimes_\Lambda\Tr_dX\rar&\cdots\rar&M^d\otimes_\Lambda\Tr_dX\rar&N\otimes_\Lambda\Tr_dX\rar&0
    \end{tikzcd}
  \end{equation*}
  at $M^i\otimes_\Lambda\Tr_dX$ is isomorphic to
  $\Tor_{d-i}^\Lambda(N,\Tr_dX)\cong\Ext_\Lambda^i(X,N)=0$.
\end{proof}

Consider the functors
\begin{equation*}
  \tau_d:=D\Tr_d\colon\smod\Lambda\to\ismod\Lambda\quad\text{and}\quad\tau_d^-:=\Tr_dD\colon\ismod\Lambda\to\smod\Lambda
\end{equation*}
which are higher analogues of the Auslander--Reiten translation and its inverse.
Note that $\tau_d$ is right adjoint to $\tau_d^-$, see Theorem 2.3.1(1) in
\cite{MR2298820}.

\begin{theorem}[(see Theorem 2.3 in \cite{MR2298819})]
  \label{taud_eq}
  Let $\M$ be a $d$-cluster-tilting subcategory of $\mmod\Lambda$. Then,
  \begin{equation*}
    \begin{tikzcd}[column sep=small]
      \tau_d\colon\sM\rar&\iM
    \end{tikzcd}\quad\text{and}\quad\begin{tikzcd}
      \tau_d^{-}\colon\iM\rar&\sM
    \end{tikzcd}
  \end{equation*}
  are mutually inverse equivalences.
\end{theorem}
\begin{proof}
  In view of \ref{Tr_equivalence} it is enough to show that for all $M\in\M$ the
  objects $\tau_d M$ and $\tau_d^-M$ also lie in $\M$. We claim that for each
  $i\in\set{1,\dots,d-1}$ and for all $N\in\M$ there is an isomorphism
  \begin{equation*}
    \Ext_\Lambda^{d-i}(N,\tau_d M)\cong D\Ext_\Lambda^i(M,N)=0.
  \end{equation*}
  Indeed, let
  \begin{equation*}
    \begin{tikzcd}[column sep=small]
      P^{-d}\rar&\cdots\rar&P^{-1}\rar&P^0\rar&M\rar&0
    \end{tikzcd}
  \end{equation*}
  be a minimal projective resolution of $M$. By definition, there is a complex
  \begin{equation*}
    \begin{tikzcd}[column sep=small]
      0\rar&\tau_d
      M\rar&D(P^{-d},\Lambda)\rar&\cdots\rar&D(P^{-1},\Lambda)\rar&D(P^{0},\Lambda)
    \end{tikzcd}
  \end{equation*}
  which is exact. In fact, for all $i\in\set{1,\dots,d-1}$ the cohomology of the
  above complex at $D\Hom_\Lambda(P^{-i},\Lambda)$ is isomorphic to
  $D\Ext_\Lambda^i(M,\Lambda)$ and hence vanishes since $M$ and $\Lambda$ belong
  to $\M$. We conclude that this sequence gives an injective coresolution of
  $\tau_dM$. Note that there is an isomorphism of complexes
  \begin{equation*}
    \begin{tikzcd}[column sep=small]
      (N,D(P^{-d},\Lambda))\rar\dar&\cdots\rar&(N,D(P^{-1},\Lambda))\rar\dar&(N,D(P^{0},\Lambda))\dar\\
      D(P^{-d},N)\rar&\cdots\rar&D(P^{-1},N)\rar&D(P^0,N).
    \end{tikzcd}
  \end{equation*}
  By definition, for each $i\in\set{1,\dots,d-1}$ the cohomology of the top row
  at
  \begin{equation*}
    \Hom_\Lambda(N,D\Hom_\Lambda(P^{-i},\Lambda))
  \end{equation*}
  is isomorphic to $\Ext_\Lambda^{d-i}(N,\tau_dM)$. Also, since $D$ is an exact
  functor, the cohomology of the bottom row at
  \begin{equation*}
    D\Hom_\Lambda(P^{-i},N)
  \end{equation*}
  is isomorphic to $\Ext_\Lambda^i(M,N)$. The claim follows. Dually, one can
  show that $\tau_d^-M\in\M$. This finishes the proof of the theorem.
\end{proof}

The following result, a higher analogue of Auslander's defect formula, is given
as Lemma 3.2 in \cite{MR2298819} where it is proven as a consequence of the
higher Auslander--Reiten duality formula.

\begin{theorem}[(Auslander's defect formula)] \label{defect_formula} Let $\M$ be
  a $d$-cluster-tilting subcategory of $\mmod\Lambda$. Then, there is a
  bifunctorial isomorphism between
  \begin{equation*}
    \begin{tikzcd}[column sep=small, row sep=tiny]
      \KdEx(\M)^\op\times \sM\rar&\mmod R\\
      (\delta,X)\rar[mapsto]&D\delta^*(X)
    \end{tikzcd}\qquad\text{and}\qquad
    \begin{tikzcd}[column sep=small, row sep=tiny]
      \KdEx(\M)^\op\times \sM\rar&\mmod R\\
      (\delta,X)\rar[mapsto]&\delta_*(\tau_dX).
    \end{tikzcd}
  \end{equation*}
\end{theorem}
\begin{proof}
  The proof is an adaptation of the proof of the main theorem in
  \cite{MR1987342}. First, we observe that $\delta^*(\Lambda)=0$, hence
  $\delta^*$ induces a functor on $\sM$. Let $X\in\M$ and
  \begin{equation*}
    \begin{tikzcd}[column sep=small]
      P^{-d}\rar&\dots\rar&P^{-1}\rar&P^0\rar&X\rar&0
    \end{tikzcd}
  \end{equation*}
  be a minimal projective resolution of $X$. Let $M\in\M$. Applying the functor
  $M\otimes_\Lambda\Hom_\Lambda(-,\Lambda)$ to this sequence yields a
  commutative diagram
  \begin{equation*}
    \resizebox{\textwidth}{!}{    
      \begin{tikzcd}[column sep=tiny, ampersand replacement=\&]
        \&\&M\otimes_\Lambda (P^0,\Lambda)\rar\dar\&M\otimes_\Lambda(P^{-1},\Lambda)\dar\rar\&\cdots\rar\&M\otimes_\Lambda (P^{-d},\Lambda)\rar\dar\&M\otimes_\Lambda\Tr_dX\dar[equals]\rar\&0\\
        0\rar\&(X,M)\rar\&(P^0,M)\rar\&(P^{-1},M)\rar\&\cdots\rar\&(P^{-d},M)\rar\&M\otimes_\Lambda\Tr_dX\rar\&0
      \end{tikzcd}}
  \end{equation*}
  in which the vertical arrows are the canonical isomorphisms. Note also that
  the bottom row of this diagram is exact, since by definition its cohomology at
  $\Hom_\Lambda(P^{-i},M)$ is isomorphic to $\Ext^i_\Lambda(X,M)$ which vanishes
  for $i\in\set{1,\dots,d-1}$.

  Let
  \begin{equation*}
    \begin{tikzcd}[column sep=small]
      \delta\colon&0\rar&L\rar{f}&M^1\rar&\cdots\rar&M^d\rar&N\rar&0
    \end{tikzcd}
  \end{equation*}
  be a $d$-exact sequence in $\M$. We obtain the following commutative diagram
  \begin{center}
    \resizebox{\textwidth}{!}{
      \begin{tikzcd}[column sep=small, row sep=small, ampersand replacement=\&]
        \&0\dar\&0\dar\&0\dar\&\&0\dar\&\Ker(f\otimes_\Lambda \Tr_d X)\dar[tail]\\
        0\rar\&(X,L)\rar\dar\&(P^0,L)\rar\dar\&(P^{-1},L)\rar\dar\&\cdots\rar\&(P^{-d},L)\rar\dar\&L\otimes_\Lambda \Tr_d X\rar\dar\&0\\
        0\rar\&(X,M^1)\rar\dar\&(P^0,M^1)\rar\dar\&(P^{-1},M^1)\rar\dar\&\cdots\rar\&(P^{-d},M^1)\rar\dar\&M^1\otimes_\Lambda \Tr_d X\rar\dar\&0\\
        \&\vdots\dar\&\vdots\dar\&\vdots\dar\&\&\vdots\dar\&\vdots\dar\\
        0\rar\&(X,M^d)\rar\dar\&(P^0,M^d)\rar\dar\&(P^{-1},M^d)\rar\dar\&\cdots\rar\&(P^{-d},M^d)\rar\dar\&M^d\otimes_\Lambda \Tr_d X\rar\dar\&0\\
        0\rar\&(X,N)\rar\dar[two heads]\&(P^0,N)\rar\dar\&(P^{-1},N)\rar\dar\&\cdots\rar\&(P^{-d},N)\rar\dar\&N\otimes_\Lambda \Tr_d X\rar\dar\&0\\
        \&\delta^*(X)\&0\&0\&\&0\&0
      \end{tikzcd}
    }
  \end{center}
  with exact rows and columns (the exactness of the rightmost column follows
  from \ref{tensor_Trd_is_right_d-exact}). A Snake Lemma-type diagram chase (for
  example, one can use the Salamander Lemma 1.7 in \cite{MR2909639}) shows that
  there is a functorial isomorphism $\delta^*(X)\cong \Ker(f\otimes_\Lambda
  \Tr_d X)$ as functors on $\sM$. Finally, we have a sequence of isomorphisms
  \begin{align*}
    D\delta^*(X)&\cong D\Ker(f\otimes_\Lambda \Tr_d X)\\
                &\cong\Coker D(f\otimes_\Lambda \Tr_d X)\\
                &\cong\Coker \Hom_\Lambda(f,D\Tr_d X)=\delta_*(\tau_dX).
  \end{align*}
  It is straightforward to verify that this ismorphism is functorial in
  $\delta$.
\end{proof}

As an immediate consequence of \ref{defect_formula}, we obtain Iyama's higher
Auslander--Reiten duality formula given as Theorem 2.3.1 in \cite{MR2298819}.

\begin{corollary}[(Higher Auslander--Reiten duality formula)]
  \label{ar_formula} Let $\M$ be a $d$-cluster-tilting subcategory of
  $\mmod\Lambda$. Then, for $X,Y\in\M$ there is a bifunctorial isomorphism
  \begin{equation*}
    D\sHom_\Lambda(X,Y)\cong \Ext_\Lambda^d(Y,\tau_dX).
  \end{equation*}
\end{corollary}
\begin{proof}
  Let $M^d\to Y$ be an epimorphism with $M^d$ a projective $\Lambda$-module. By
  \ref{minimal} and \ref{d-ct_a2}, there exists a $d$-exact sequence in $\M$
  \begin{equation*}
    \begin{tikzcd}[column sep=small]
      \delta\colon&0\rar&L\rar&M^1\rar&\cdots\rar&M^d\rar&Y\rar&0.
    \end{tikzcd}
  \end{equation*}
  Hence, by \ref{d-rigid_long_exact} and \ref{taud_eq} the functor
  $\Hom_\Lambda(-,\tau_dX)$ induces an exact sequence
  \begin{equation*}
    \begin{tikzcd}[column sep=small]
      \Hom_\Lambda(M^d,\tau_dX)\rar&\Hom_\Lambda(L,\tau_d
      X)\rar&\Ext_\Lambda^d(Y,\tau_d X)\rar&\Ext_\Lambda^d(M^d,\tau_d X)=0
    \end{tikzcd}
  \end{equation*}
  Therefore $\delta_*(\tau_dX)\cong \Ext_\Lambda^d(Y,\tau_dX)$. On the other
  hand, it follows from the lifting property of projective $\Lambda$-modules
  that $\delta^*(X)\cong\sHom_\Lambda(X,Y)$. The claim now follows from
  \ref{defect_formula}.
\end{proof}

We also obtain the following important consequence of the defect formula.

\begin{corollary}[(\cf Corollary IV.4.4 in \cite{MR1476671})]
  \label{factorization_defect_formula} Let $\M$ be a $d$-cluster-tilting
  subcategory of $\mmod\Lambda$ and
  \begin{equation*}
    \begin{tikzcd}[column sep=small]
      \delta\colon&0\rar&L\rar{f}&M^1\rar&\cdots\rar&M^d\rar{g}&N\rar&0
    \end{tikzcd}
  \end{equation*}
  a $d$-exact sequence in $\M$. For an object $X\in\M$, the following statements
  are equivalent.
  \begin{enumerate}
  \item\label{it:FTb} Every morphism $L\to X$ factors through $f$.
  \item\label{it:FTa} Every morphism $\tau_d^-X\to N$ factors through $g$.
  \end{enumerate}
\end{corollary}
\begin{proof}
  By definition, the functors $\Hom_\Lambda(-,X)$ and
  $\Hom_\Lambda(\tau_d^-X,-)$ induce exact sequences
  \begin{equation*}
    \begin{tikzcd}[column sep=small, row sep=tiny]
      \Hom_\Lambda(M^1,X)\rar&\Hom_\Lambda(L,X)\rar&\delta_*(X)\rar&0
    \end{tikzcd}
  \end{equation*}
  and
  \begin{equation*}
    \begin{tikzcd}[column sep=small, row sep=tiny]      
      \Hom_\Lambda(\tau_d^-X,M^d)\rar&\Hom_\Lambda(\tau_d^-X,N)\rar&\delta^*(\tau_d^-X)\rar&0
    \end{tikzcd}
  \end{equation*}
  in $\mmod R$. By the dual of \ref{defect_formula}, we have $\delta_*(X)=0$ if
  and only if $\delta^*(\tau_d^-X)=0$. The claim follows.
\end{proof}

\section{Morphisms determined by objects}
\label{sec:morphisms_determined_by_objects}

In this section we show that $d$-cluster-tilting subcategories of $\mmod\Lambda$
have right determined morphisms.


\begin{definition}[(see Section 1 in \cite{MR0480688})]
  Let $\X$ be a subcategory of $\mmod\Lambda$ and $X\in\X$. A morphism $g\colon
  M\to N$ in $\X$ is \emph{right $X$-determined in $\X$} if for every morphism
  $h\colon V\to N$ in $\X$ such that
  \begin{equation*}
    \Img\Hom_\Lambda(X,h)\subseteq\Img\Hom_\Lambda(X,g)
  \end{equation*}
  we have that $h$ factors through $g$. In such a case, we call $X$ a
  \emph{right determiner\footnote{We use Ringel's terminology, see
      \cite{MR3161362}.} of $g$ in $\X$}.
\end{definition}

\begin{remark}
  If $g\colon M\to N$ and $g'\colon M'\to N$ are right minimal right
  $X$-determined morphisms in some subcategory $\X$ of $\mmod\Lambda$ such that
  $\Img\Hom_{\Lambda}(X,g)=\Img\Hom_{\Lambda}(X,g')$, then there exists an
  isomorphism $h\colon M\to M'$ such that $g=g'\circ h$. In \ref{auslander2} we
  prove an existence theorem of right $X$-determined morphisms which complements
  this statement.
\end{remark}

The following result shows that every morphism in a $d$-cluster-tilting
subcategory of $\mmod\Lambda$ has an explicit determiner. In our setting, this
result gives further information than Corollary 3.8 in \cite{MR3183886} which
only establishes the existence of a determiner.

\begin{theorem}[(\cf Corollary XI.1.4 in \cite{MR1476671})]
  \label{auslander1} Let $\M$ be a $d$-cluster-tilting subcategory of
  $\mmod\Lambda$ and
  \begin{equation*}
    \begin{tikzcd}[column sep=small]
      0\rar&L\rar{f}&M^1\rar&\cdots\rar&M^d\rar{g}&N
    \end{tikzcd}
  \end{equation*}
  a left $d$-exact sequence in $\M$. Then, $\Lambda\oplus\tau_d^-L$ is a right
  determiner of $g$ in $\M$. If $g$ is an epimorphism, then $\tau_d^- L$ is a
  right determiner of $g$ in $\M$.
\end{theorem}
\begin{proof}
  Firstly, let $h\colon V\to N$ be a morphism in $\M$ such that
  \begin{equation*}
    \Img\Hom_\Lambda(\Lambda\oplus\tau_d^-
    L,h)\subseteq\Img\Hom_\Lambda(\Lambda\oplus\tau_d^-L,g),
  \end{equation*}
  that is for every morphism $h'\colon \Lambda\oplus\tau_d^-L\to V$ the
  composite $h\circ h'$ factors through $g$. Applying \ref{existence_d-pullback}
  and \ref{d-pullback_left_exact_d-sequence}, we obtain the solid $d$-pullback
  diagram
  \begin{equation*}
    \begin{tikzcd}
      &&&&\Lambda\oplus\tau_d^-L\rar{h'}\arrow[bend right=50,dotted]{dd}\dar[dotted]&V\dar[equals]\\
      0\rar&L\rar{f'}\dar[equals]&U^1\rar\dar&\cdots\rar&U^d\rar{g'}\dar&V\dar{h}\rar[dotted]&0\\
      0\rar&L\rar{f}&M^1\rar&\cdots\rar&M^d\rar{g}&N
    \end{tikzcd}
  \end{equation*}
  in which the middle row is also a left $d$-exact sequence in $\M$. Then, by
  the factorisation property of $d$-pullback diagrams we conclude that every
  $h'$ as above factors through $g'$. Since $g'$ is an epimorphism if and only
  if every morphism $\Lambda\to V$ factors through $g'$, we deduce from the
  previous discussion that $g'$ is indeed an epimorphism. Applying
  \ref{d-ct_a2}, we conclude the middle row is a $d$-exact sequence in $\M$.

  Secondly, since every morphism $\tau_d^-L\to V$ factors through $g'$, by
  \ref{factorization_defect_formula} every morphism $L\to L$ factors through
  $f'$. In particular, $f'$ is a split monomorphism and by
  \ref{contractible_iff_split_morphisms} we have that $g'$ is a split
  epimorphism as well. Therefore $h$ factors through $g$. This shows that
  $\Lambda\oplus\tau_d^-L$ is a right determiner of $g$.
  
  Suppose now that $g$ is an epimorphism. Then, in the diagram above, the bottom
  row is a $d$-exact sequence by \ref{d-ct_a2}. Consequently, by
  \ref{props_d-pullback} the middle row is also a $d$-exact sequence. Then, the
  argument in the previous paragraph shows that $h$ factors through $g$.
  Therefore $\tau_d^-L$ is a right determiner of $g$ in this case.
\end{proof}

The following statement is one of the main results in this section. Although it
is a special case of Proposition 3.9 in \cite{MR3183886}, we give a more
constructive proof using Auslander's defect formula.

\begin{theorem}
  \label{auslander2} Let $\M$ be a $d$-cluster-tilting subcategory of
  $\mmod\Lambda$ and $X,N\in\M$. Then, for each $\End_\Lambda(X)$-submodule
  $H\subseteq\Hom_\Lambda(X,N)$ there exists a morphism $g\colon M\to N$ in $\M$
  which is right $X$-determined in $\M$ such that $\Img\Hom_\Lambda(X,g)=H$.
\end{theorem}

In the case $d=1$, the above theorem specialises to a classical result due to
Auslander.

\begin{corollary}[(see Theorem XI.3.6 in \cite{MR1476671})]
  \label{auslander2-coro}
  Let $X,N\in\mmod\Lambda$. Then, for each $\End_\Lambda(X)$-submodule
  $H\subseteq\Hom_\Lambda(X,N)$ there exists a morphism $g\colon M\to N$ in
  $\mmod\Lambda$ which is right $X$-determined in $\mmod\Lambda$ and such that
  $\Img\Hom_\Lambda(X,g)=H$.
\end{corollary}

\begin{remark}
  It is also possible to deduce \ref{auslander2} from
  \ref{auslander2-coro} as follows. Given a $d$-cluster-tilting subcategory $\M$
  of $\mmod\Lambda$, $X,N\in\M$ and an $\End_\Lambda(X)$-submodule
  $H\subseteq\Hom_\Lambda(X,N)$, \ref{auslander2-coro} yields a morphism
  $g\colon Y\to N$ in $\mmod\Lambda$ which is right $X$-determined in
  $\mmod\Lambda$ and such that $\Img\Hom_\Lambda(X,g)=H$. Let $h\colon M\to Y$
  be a right $\M$-approximation of $Y$. Then, the morphism $g':=g\circ h$ is
  right $X$-determined in $\M$ and satisfies $\Img\Hom_\Lambda(X,g')=H$.
\end{remark}

Before giving the proof of \ref{auslander2} we prove the following general
result.

\begin{proposition}[(\cf Proposition XI.3.2 \cite{MR1476671})]
  \label{auslander2_prop} Let $\M$ be a $d$-cluster-tilting subcategory of
  $\mmod\Lambda$, $X\in\M$ and
  \begin{equation*}
    \begin{tikzcd}[column sep=small]
      \delta\colon&0\rar&L\rar{f}&M^1\rar&\cdots\rar&M^d\rar{g}&N\rar&0
    \end{tikzcd}
  \end{equation*}
  a $d$-exact sequence in $\M$. Then, there exists a $d$-pushout diagram
  \begin{equation*}
    \begin{tikzcd}[column sep=small]
      \delta\colon&0\rar&L\rar{f}\dar&M^1\rar\dar&\cdots\rar&M^d\rar{g}\dar&N\rar\dar[equals]&0\\
      \varepsilon\colon&0\rar&L_X\rar{f_X}&M_X^1\rar&\cdots\rar&M_X^d\rar{g_X}&N\rar&0
    \end{tikzcd}
  \end{equation*}
  such that
  \begin{enumerate}
  \item\label{it:g_X_right_X-determined} The morphism $g_X$ is right
    $X$-determined in $\M$.
  \item\label{it:H_gX-H_g} We have
    $\Img\Hom_\Lambda(X,g_X)=\Img\Hom_\Lambda(X,g)$.
  \end{enumerate}
\end{proposition}
\begin{proof}
  Let $h\colon L\to L_X$ be a left $(\add\tau_dX)$-approximation of $L$. Thus,
  the composite
  \begin{equation*}
    \begin{tikzcd}
      \Hom_\Lambda(L_X,\tau_dX)\rar{(h,\tau_dX)}&\Hom_\Lambda(L,\tau_dX)\rar&\delta_*(\tau_dX)
    \end{tikzcd}
  \end{equation*}
  yields a surjection from a projective $\End_\Lambda(\tau_dX)$-module onto
  $\delta_*(\tau_dX)$. Also, by the dual of \ref{existence_d-pullback} and
  \ref{d-pullback_left_exact_d-sequence} there exists a $d$-pushout diagram
  \begin{equation*}
    \begin{tikzcd}
      \delta\dar{\phi}&0\rar&L\rar{f}\dar{h}&M^1\rar\dar&\cdots\rar&M^d\rar{g}\dar&N\rar\dar[equals]&0\\
      \varepsilon&0\rar&L_X\rar{f_X}&M_X^1\rar&\cdots\rar&M_X^d\rar{g_X}&N\rar&0.
    \end{tikzcd}
  \end{equation*}

  \ref{it:g_X_right_X-determined} By \ref{auslander1}, the morphism $g_X$ is
  right $\tau_d^{-}L_X$-determined and since $L_X\in\add\tau_dX$, it follows
  from \ref{taud_eq} that $g_X$ is right $X$-determined in $\M$.

  \ref{it:H_gX-H_g} Since $g$ factors through $g_X$ we have
  $\Img\Hom_\Lambda(X,g)\subseteq\Img\Hom_\Lambda(X,g_X)$. By applying the
  functor $\Hom_\Lambda(X,-)$ to the above diagram and using the snake lemma, we
  obtain a commutative diagram with exact rows
  \begin{equation*}
    \begin{tikzcd}[column sep=large]
      &&\frac{\Img(X,g_X)}{\Img(X,g)}\dar[tail]\\
      \Hom_\Lambda(X,M^d)\rar{(X,g)}\dar&\Hom_\Lambda(X,N)\rar\dar[equals]&\delta^*(X)\rar\dar[two heads]{\phi^*(X)}&0\\
      \Hom_\Lambda(X,M_X^d)\rar{(X,g_X)}&\Hom_\Lambda(X,N)\rar&\varepsilon^*(X)\rar&0
    \end{tikzcd}
  \end{equation*}
  and where the left column is a short exact sequence. Hence it is sufficient to
  show that $\phi^*(X)$ is a monomorphism or, equivalently by
  \ref{defect_formula}, that $\phi_*(\tau_dX)$ is an epimorphism. Applying the
  functor $\Hom_\Lambda(-,\tau_dX)$ to the same diagram as before yields a
  commutative diagram with exact rows
  \begin{equation*}
    \begin{tikzcd}[column sep=small]
      \Hom_\Lambda(M_X^1,\tau_dX)\rar\dar&\Hom_\Lambda(L_X,\tau_dX)\rar\dar[swap]{(h,\tau_dX)}\drar[two heads]&\varepsilon_*(\tau_dX)\rar\dar{\phi_*(\tau_dX)}&0\\
      \Hom_\Lambda(M^1,\tau_dX)\rar&\Hom_\Lambda(L,\tau_dX)\rar&\delta_*(\tau_dX)\rar&0
    \end{tikzcd}
  \end{equation*}
  where the diagonal arrow is an epimorphism since it is a projective cover.
  Hence $\phi_*(\tau_dX)$ is an epimorphism and the claim follows.
\end{proof}

As a consequence of \ref{auslander2_prop} we obtain the following special case
of \ref{auslander2}.

\begin{corollary}[(\cf Corollary XI.3.4 \cite{MR1476671})]
  \label{proj_auslander2} Let $\M$ be a $d$-cluster-tilting subcategory of
  $\mmod\Lambda$ and $X,N\in\M$. Then, for each $\End_\Lambda(X)$-submodule
  $H\subseteq\Hom_\Lambda(X,N)$ such that $\mathcal{P}(X,N)\subseteq H$ there
  exists an epimorphism $g_{X,H}\colon M_{X,H}\to N$ in $\M$ which is right
  $X$-determined in $\M$ and such that $\Img\Hom_\Lambda(X,g_{X,H})=H$.
\end{corollary}
\begin{proof}
  Observe that, since $H$ is a finitely generated $\End_\Lambda(X)$-module,
  there exist $X'\in\add X$, $g'\colon X'\to N$ such that
  $\Img\Hom_\Lambda(X,g')=H$. Let $p\colon P\to N$ be a projective cover. Then,
  $g:=[g'\ p]\colon X'\oplus P\to N$ is an epimorphism and
  \begin{equation*}
    \Img\Hom_\Lambda(X,g)=\Img\Hom_\Lambda(X,g')=H
  \end{equation*}
  since $\mathcal{P}(X,N)\subseteq H$. Taking into account that every
  epimorphism in $\M$ gives rise to a $d$-exact sequence (see \ref{minimal} and
  \ref{d-ct_a2}), we conclude from \ref{auslander2_prop} that there exists a
  right $X$-determined epimorphism $g_{X,H}\colon M_{X,H}\to N$ in $\M$ such
  that $\Img\Hom_\Lambda(X,g_{X,H})=H$, which is what we needed to show.
\end{proof}

We reduce the proof of \ref{auslander2} to \ref{proj_auslander2}. In what
follows we adapt the approach used in the proof of Lemma XI.3.7 in
\cite{MR1476671} to our setting. Fix a module $N\in\mmod\Lambda$ and let
$(\mmod\Lambda)/N$ be the comma category over $N$, see for example Section I.2
in \cite{MR1476671}. We also recall that there is an equivalence relation on
$(\mmod\Lambda)/N$ given by $f\sim f'$ if there exist morphisms $g\colon f\to
f'$ and $g'\colon f'\to f$ in $(\mmod\Lambda)/N$. We denote the equivalence
class of $f$ in $(\mmod\Lambda)/N$ by $[f]$. There is a partial order on the set
of equivalence classes in $(\mmod\Lambda)/N$ given by $[f]\leq [f']$ if there
exists a morphism $g\colon f\to f'$.




\begin{lemma}
  \label{technical_lemma2} Let $\M$ be a $d$-cluster-tilting subcategory of
  $\mmod\Lambda$, $X,N\in\M$ and $H$ an $\End_\Lambda(X)$-submodule of
  $\Hom_\Lambda(X,N)$. Then, there exists a morphism $u\colon N_H\to N$ in $\M$
  satisfying the following properties:
  \begin{enumerate}
  \item\label{it:i-a} $\Hom_{\Lambda}(X,u)(\mathcal{P}(X,N_H))\subseteq
    H\subseteq\Img\Hom_\Lambda(X,u)$.
  \item\label{it:tech_prop22} For every $h\colon V\to N$ in $\M$ such that
    $\Hom_\Lambda(X,h)(\mathcal{P}(X,V))\subseteq H$, the morphism $h$ factors
    through $u$.
  \end{enumerate}
\end{lemma}
\begin{proof}
  Our method of proof is similar to that of proof of Lemma XI.3.7 in
  \cite{MR1476671}. Let $\cat{S}$ be the set of submodules $L$ of $N$ such that
  $\mathcal{P}(X,L)\subseteq H$. Note that if $L_1$ and $L_2$ lie in $\cat{S}$, then the
  sum $L_1+L_2$ also lies in $\cat{S}$ for we have
  \begin{equation*}
    \mathcal{P}(X,L_1+L_2)\subseteq\mathcal{P}(X,L_1)+\mathcal{P}(X,L_2)\subseteq H+H=H.
  \end{equation*}
  Moreover, given that $N$ has finite length, the set $\cat{S}$ has a maximal
  element, which we denote by $L_H$. Let $u\colon N_H\to L_H$ be a right
  $\M$-approximation of $L_H$. By construction, $u$ satisfies the first
  inclusion in property \ref{it:i-a}.

  Let us prove statement \ref{it:tech_prop22}. Let $h\colon V\to N$ be a
  morphism in $\M$ satisfying $\Hom_\Lambda(X,h)(\mathcal{P}(X,V))\subseteq H$. Note that $\Img
  h$ lies in $\cat{S}$ since
  \begin{equation*}
    \mathcal{P}(X,\Img h)\subseteq \Hom_\Lambda(X,h)(\mathcal{P}(X,V))\subseteq H.
  \end{equation*}
  The maximality of $L_H$ then implies the inclusion $\Img h\subseteq L_H$.
  Finally, since $u$ is a right $\M$-approximation of $L_H$, we conclude that
  $h$ factors through $u$ as desired. This proves \ref{it:tech_prop22}.
  
  
  We finish the proof by proving that $H\subseteq\Img\Hom_\Lambda(X,u)$, which
  is the remaining part of statement of \ref{it:i-a}. Let $h''\colon X'\to N$ be
  a morphism with $X'\in\add X$ and such that $\Img\Hom_{\Lambda}(X,h'')=H$
  (such a morphism $h''$ exists since $H$ is a finitely generated
  $\End_{\Lambda}(X)$-module). Clearly, $h''$ satisfies the assumptions in part
  \ref{it:tech_prop22} whence it factors through $u$. Therefore
  \begin{equation*}
    \Img\Hom_{\Lambda}(X,h'')=H\subseteq\Img\Hom_{\Lambda}(X,u),
  \end{equation*}
  as required.
\end{proof}


We are ready to prove \ref{auslander2}.

\begin{proof}[of \ref{auslander2}]
  Let $u\colon N_H\to N$ be the morphism given by \ref{technical_lemma2} and set
  $H':=\Hom_\Lambda(X,u)^{-1}(H)$. Since
  $\Hom_\Lambda(X,u)(\mathcal{P}(X,N_H))\subseteq H$ (see
  \ref{technical_lemma2}), by \ref{proj_auslander2} there exists a right
  $X$-determined morphism $g_{X,H'}\colon M_{X,H'}\to N_H$ such that
  $\Img\Hom_\Lambda(X,g_{X,H'})=H'$. Define $g:=u\circ g_{X,H'}$. Since by
  \ref{technical_lemma2} we have $H\subseteq\Img\Hom_{\Lambda}(X,u)$, it follows
  that
  \begin{equation*}
    \Img\Hom_\Lambda(X,g)=\Hom_{\Lambda}(X,u)(H')=H
  \end{equation*}
  It remains to show that $g$ is right $X$-determined in $\M$. Let $h\colon V\to
  N$ be a morphism such that $\Img(X,h)\subseteq H=\Img\Hom_\Lambda(X,g)$. Then,
  by \ref{technical_lemma2}\ref{it:tech_prop22} the morphism $h$ factors through
  $u$ as $h=u\circ v$, say. It follows that $\Img\Hom_\Lambda(X,v)\subseteq
  H'=\Img\Hom_\Lambda(X,g_{X,H'})$. Finally, since $g_{X,H'}$ is right
  $X$-determined in $\M$ we have that $v$ factors through $g_{X,H'}$ and hence
  $h$ factors through $g$. This shows that $g$ is right $X$-determined in $\M$.
  This finishes the proof of the theorem.
\end{proof}

\section{$d$-almost split sequences}
\label{sec:d-almost split_sequences}

In this section we establish the existence of $d$-almost split sequences in
$d$-cluster-tilting subcategories of $\mmod\Lambda$.

Let $\M$ be a subcategory of $\mmod\Lambda$. Recall that a morphism $g\colon
M\to N$ in $\M$ is \emph{right almost split in $\M$} if it is not a split
epimorphism and every morphism $L\to N$ in $\M$ which is not a split epimorphism
factors through $g$.

\begin{remark}
  Let $\M$ be a subcategory of $\mmod\Lambda$ containing $\Lambda$ and $g\colon
  M\to N$ a right almost split morphism in $\M$. Note that if $N$ is not
  projective, then its projective cover factors through $g$ which implies that $g$ is
  an epimorphism.
\end{remark}

The following result is classical.

\begin{proposition}
  \label{almost split-right_determined} Let $\M$ be a subcategory of
  $\mmod\Lambda$ and $g\colon M\to N$ a morphism in $\M$ such that $N$ is
  indecomposable. Then, the following statements are equivalent.
  \begin{enumerate}
  \item\label{it:right_almost split} $g$ is right almost split in $\M$.
  \item\label{it:right_N-determined} $g$ is right $N$-determined in $\M$ and
    $\Img\Hom_\Lambda(N,g)=\rad\Hom_\Lambda(N,N)$.
  \end{enumerate}
\end{proposition}
\begin{proof}
  Recall that since $N$ is indecomposable $\rad\Hom_\Lambda(N,N)$ consists
  precisely of the non-invertible morphisms in $\Hom_\Lambda(N,N)$.

  We prove that \ref{it:right_almost split} implies \ref{it:right_N-determined}.
  We begin by showing that $\Img\Hom_\Lambda(N,g)=\rad\Hom_\Lambda(N,N)$. It is
  clear from the definition of almost split morphism that
  $\rad\Hom_{\Lambda}(N,N)\subseteq\Img\Hom_\Lambda(N,g)$. Conversely, let
  $h\colon N\to M$ be a morphism and suppose that $g\circ
  h\notin\rad\Hom_{\Lambda}(N,N)$. Hence $g\circ h$ is invertible and therefore
  there exists a morphism $i\colon N\to N$ such that $g\circ h\circ i=1_N$. This
  contradicts the fact that $g$ is not a split epimorphism and the claim
  follows.

  Let $h\colon V\to N$ be a morphism in $\M$ such that
  $\Img\Hom_\Lambda(N,h)\subseteq\Img(N,g)$. We claim that $h$ is not a split
  epimorphism. Indeed, suppose otherwise; then there exists a morphism $i\colon
  N\to V$ such that $h\circ i=1_N$. Hence
  $1_N\in\Img\Hom_\Lambda(N,h)\subseteq\rad\Hom_\Lambda(N,N)$, a contradiction.
  Therefore $h$ is not a split epimorphism and, given that $g$ is right almost
  split, $h$ factors through $g$. This shows that $g$ is right $N$-determined in
  $\M$.

  We prove that \ref{it:right_N-determined} implies \ref{it:right_almost split}.
  Let $h\colon V\to N$ be a morphism in $\M$ which is not a split epimorphism.
  One can show, as above, that
  $\Img\Hom_\Lambda(N,h)\subseteq\rad\Hom_\Lambda(N,N)=\Img\Hom_\Lambda(N,g)$.
  Finally, since $g$ is right $N$-determined we have that $h$ factors through
  $g$, which is what we needed to prove.
\end{proof}

As a consequence of \ref{auslander2} and \ref{almost split-right_determined}, we
deduce the existence of right almost split morphisms in $d$-cluster-tilting
subcategories of $\mmod\Lambda$.

\begin{corollary}
  \label{existence_right_almost split_morphisms} Let $\M$ be a
  $d$-cluster-tilting subcategory of $\mmod\Lambda$. Then, for every
  indecomposable object $N\in\M$ there exists a morphism $M\to N$ in $\M$ which
  is right almost split in $\M$.
\end{corollary}
\begin{proof}
  Let $H:=\rad\End_\Lambda(N)$. Then, \ref{auslander2} yields a morphism
  $g\colon M\to N$ in $\M$ which is right $N$-determined in $\M$ and such that
  $\Img\Hom_\Lambda(N,g)=\rad\End_\Lambda(N,N)$. By \ref{almost
    split-right_determined}, the morphism $g$ is right almost split in $\M$.
\end{proof}

We recall the definition of a $d$-almost split sequence.

\begin{definition}[(see Definition 3.1 in \cite{MR2298819})]
  Let $\M$ be a $d$-cluster-tilting subcategory of $\mmod\Lambda$ and
  \begin{equation*}
    \begin{tikzcd}[column sep=small]
      \delta\colon&0\rar&L\rar{f}&M^1\rar&\cdots\rar&M^d\rar{g}&N\rar&0
    \end{tikzcd}
  \end{equation*}
  a $d$-exact sequence in $\M$. We say that $\delta$ is a \emph{$d$-almost split
    sequence in $\M$} if $f$ is left minimal and left almost split in $\M$, $g$
  is right minimal and right almost split in $\M$ and for each
  $i\in\set{1,\dots,d-1}$ the morphism $M^i\to M^{i+1}$ lies in the Jacobson
  radical of $\mmod\Lambda$.
\end{definition}

We are ready to prove the main theorem of this section.

\begin{theorem}[(see Theorem 3.3.1 in \cite{MR2298819})]
  \label{existence_d-almost split-sequences}
  Let $\M$ be a $d$-cluster-tilting subcategory of $\mmod\Lambda$. Then, for
  each indecomposable non-projective $\Lambda$-module $N\in\M$ there exists a
  $d$-almost split sequence in $\M$ of the form
  \begin{equation*}
    \begin{tikzcd}[column sep=small]
      0\rar&\tau_d N\rar&M^1\rar&\cdots\rar&M^d\rar&N\rar&0.
    \end{tikzcd}
  \end{equation*}
  Dually, for each indecomposable non-injective $\Lambda$-module $L\in\M$ there
  exists a $d$-almost split sequence in $\M$ of the form
  \begin{equation*}
    \begin{tikzcd}[column sep=small]
      0\rar&L\rar&M^1\rar&\cdots\rar&M^d\rar&\tau_d^{-}L\rar&0.
    \end{tikzcd}
  \end{equation*}
\end{theorem}
\begin{proof}
  Let $N\in\M$ be an indecomposable non-projective $\Lambda$-module. The case of
  an indecomposable non-injective $\Lambda$-module $L\in\M$ is dual.

  By \ref{existence_right_almost split_morphisms} there exists a right minimal
  morphism $g\colon M^d\to N$ which is right almost split in $\M$. Also, note
  that $g$ lies in the Jacobson radical of $\mmod\Lambda$ since $N$ is
  indecomposable and $g$ is not a split epimorphism. By \ref{minimal} there
  exists a left $d$-exact sequence
  \begin{equation*}
    \begin{tikzcd}[column sep=small]
      \delta\colon&0\rar&N'\rar{f}&M^1\rar&\cdots\rar&M^d\rar{g}&N\rar&0
    \end{tikzcd}
  \end{equation*}
  such that for each $i\in\set{1,\dots,d-1}$ the morphism $M^i\to M^{i+1}$ is
  right minimal. Moreover, since $g$ is right minimal, \ref{minimal-jacobson}
  implies that for each $i\in\set{1,\dots,d-1}$ the morphism $M^i\to M^{i+1}$
  lies in the Jacobson radical of $\mmod\Lambda$. Since in particular $M^1\to
  M^2$ lies in the Jacobson radical of $\mmod\Lambda$, the dual of
  \ref{minimal-jacobson} implies that $f$ is left minimal.

  We claim that $f$ is left almost split in $\M$. Indeed, let $h\colon N'\to
  N''$ be a morphism in $\M$ which is not a split monomorphism and suppose that
  $h$ does not factor through $f$. By \ref{existence_d-pullback} and
  \ref{props_d-pullback} there exists a $d$-pushout diagram
  \begin{equation*}
    \begin{tikzcd}[column sep=small]
      0\rar&N'\rar{f}\dar{h}&M^1\rar\dar&\cdots\rar&M^d\rar{g}\dar&N\rar\dar[equals]&0\\
      0\rar&N''\rar{f'}&X^1\rar&\cdots\rar&X^d\rar{g'}&N\rar&0
    \end{tikzcd}
  \end{equation*}
  in which the bottom row is a $d$-exact sequence in $\M$. Moreover, $g'$ is not
  a split epimorphism; if it were, by \ref{contractible_iff_split_morphisms} the
  morphism $f'$ would be a split monomorphism and then $h$ would factor through
  $f$ which would contradict our assumption on $h$. Therefore $g'$ factors
  through the right almost split morphism $g$. It follows that there exists a
  commutative diagram
  \begin{equation*}
    \begin{tikzcd}[column sep=small]
      0\rar&N'\rar{f}\dar{h}&M^1\rar\dar&\cdots\rar&M^d\rar{g}\dar&N\rar\dar[equals]&0\\
      0\rar&N''\rar{f'}\dar{h'}&X^1\rar\dar&\cdots\rar&X^d\rar{g'}\dar&N\rar\dar[equals]&0\\
      0\rar&N'\rar{f}&M^1\rar&\cdots\rar&M^d\rar{g}&N\rar&0
    \end{tikzcd}
  \end{equation*}
  By \ref{comparison_lemma} there exists a morphism $u\colon M^1\to N'$ such
  that $h'\circ h-1_{N'}=u\circ f$ or, equivalently, $h'\circ h=1_{N'}+u\circ
  f$. Since $f$ lies in the Jacobson radical of $\mmod\Lambda$, the morphism
  $h'\circ h$ is an isomorphism. In particular $h$ is a split monomorphism,
  which is a contradiction. This shows that $f$ is left almost split in $\M$.
  
  It remains to show that $N'\cong\tau_d N$. Since $f$ is left almost split in
  $\M$, the $\Lambda$-module $N'$ is indecomposable. Moreover, since $g$ is not
  a split epimorphism the defect $\delta^*(N)$ is non-zero. Hence, by
  \ref{defect_formula}, the defect $\delta_*(\tau_d N)$ is non-zero. Therefore
  there exists a morphism $s\colon N'\to\tau_d N$ which does not factor through
  $f$ whence $s$ is a split monomorphism. Since $N'$ is indecomposable the
  morphism $s$ is in fact an isomorphism. This finishes the proof of the
  theorem.
\end{proof}
